\documentclass[11pt]{amsproc}

\usepackage{amsmath, amssymb, amsthm, latexsym,nicefrac,setspace}
\usepackage[pagebackref,colorlinks,linkcolor=blue,citecolor=blue,urlcolor=blue,hypertexnames=true]{hyperref}
\usepackage{amsrefs}
\usepackage{thmtools}
\usepackage{mathtools}

\usepackage[usenames]{color}

\newcommand{\textoverline}[1]{$\bar{\mbox{#1}}$}

\theoremstyle{plain}
\newtheorem{theorem}{Theorem}[section]
\newtheorem{question}[theorem]{Question}
\newtheorem{lemma}[theorem]{Lemma}
\newtheorem{corollary}[theorem]{Corollary}
\newtheorem{proposition}[theorem]{Proposition}

\theoremstyle{remark}
\newtheorem{remark}[theorem]{Remark}

\input{xy}
\xyoption{all}

\author{Vadim Alekseev}
\address{Vadim Alekseev, TU Dresden, Germany}
\email{vadim.alekseev@tu-dresden.de}

\author{Andreas Thom}
\address{Andreas Thom, TU Dresden, Germany}
\email{andreas.thom@tu-dresden.de}

\title[Maximal discrete subgroups]{Maximal discrete subgroups in  unitary groups \\ of operator algebras}

\begin{document}

\onehalfspace

\begin{abstract}
We show that if a group $G$ is mixed-identity-free, then the projective unitary group of its group von Neumann algebra contains a maximal discrete subgroup containing $G$. The proofs are elementary and make use of free probability theory. In addition, we clarify the situation for $C^*$-algebras.
\end{abstract}

\maketitle

\section{Introduction}

It has been of interest since the beginning of the theory of topological groups to study discrete subgroups for familiar topological groups, such as connected Lie groups or other locally compact groups arising from arithmetic or geometry \cites{MR29918, MR374340, MR0209363,MR212021,MR2198218,MR41145,MR171781}.  Of special interest is the question if a discrete subgroup is maximal among discrete subgroups, i.e.\ when every group containing it properly is not discrete, or if, at least, it is contained in a maximal discrete subgroup. For ${\rm PU}(n)$, the Jordan--Schur theorem states that any discrete (and hence finite) subgroup has an abelian normal subgroup of uniformly bounded index. This implies that some finite subgroups are contained in maximal discrete subgroups -- we will clarify the situation in Proposition \ref{compact}. This realm of questions becomes even more interesting for non-compact groups. For example, Helling's theorem \cite{MR228437} states that a maximal discrete subgroup of ${\rm PSL}(2,\mathbb R)$ that is commensurable with ${\rm PSL}(2,\mathbb Z)$ is conjugate to a congruence subgroup $\Gamma^+_0(N)$ for some square-free integer $N$. In \cite{MR2198218}, Belolipetsky and Lubotzky provided examples of maximal discrete subgroups of the isometry group of higher-dimensional hyperbolic space arising as commensurators of certain non-arithmetic lattices. 

Typical existence results of maximal discrete subgroups usually show that the family of discrete subgroups is uniformly discrete and employ some contraction mechanism near the identity element, for example in the usual operator norm on ${\rm GL}(n,\mathbb C)$
$$\|1-uvu^{-1}v^{-1}\| \leq 2 \|1-u\| \|1-v\| \|u^{-1}\| \|v^{-1}\|.$$  
Any result like this provides a tool to construct smaller and smaller elements in a group once it gets close enough to the unit element. If the group is assumed to be discrete in addition one obtains nilpotency phenomena near the identity. More precisely, there is the following classical result of Zassenhaus \cite{MR3069692} which says in modern form: If $G$ is a semisimple Lie group there exists a neighbourhood $\Omega$ of the identity in $G$ and a constant $C > 0$, such that any discrete subgroup $\Gamma$ which is generated by $\Gamma \cap \Omega$ contains a nilpotent subgroup of index bounded by $C$. This has been generalized by Margulis to cover isometry groups of spaces of negative sectional curvature.

The situation is much different in the non-locally compact setting, but a similar reasoning works for example in the unitary group of say a $C^*$-algebra and hinges only on submultiplicativity of the norm which implies contractivity of the commutator map near the unit. Again, this implies by a Zassenhaus argument that any discrete subgroup $\Gamma$ of the unitary group of a $C^*$-algebra $A$ satisfies that the group generated by the set
$\{g \in \Gamma \mid \|1-g\|<1/2 \}$ is abelian and normal, see Theorem \ref{cstar}. Nothing is known about the concrete examples of maximal discrete subgroups.

We aim to study this phenomenon in the context of finite von Neumann algebras, where the unitary group is metrized by the $2$-norm and no contractivity of the commutator map is available in general. We will make use of free probability theory to overcome this problem in the situation of the group von Neumann algebra $LG$, and show in many cases that nevertheless the class of discrete subgroups of the projective unitary group ${\rm PU}(LG)$ containing $G$ is uniformly discrete. As a consequence, $G$ is contained in a maximal discrete subgroup. We do not know if or when $G \subset {\rm PU}(LG)$ is maximal discrete itself.

\section{Preparations using free probability}

Let $(M,\tau)$ be a finite von Neumann algebra and we will typically assume that it is a ${\rm II}_1$-factor. The $2$-norm on $M$ is defined as $\|x\|_2 = \tau(x^*x)^{1/2}$. We denote its unitary group by ${\rm U}(M)$ and its projective unitary group by ${\rm PU}(M):={\rm U}(M)/S^1$. Consider unitaries $u,v \in {\rm U}(M)$ which are  freely independent, that is to say that they lie in  freely independent subalgebras in the sense of Voiculescu \cite{MR1217253}. We set $\alpha := \tau(u)$ and $\beta := \tau(v)$. Note that because of freeness we have
$$0 = \tau((u-\alpha)(v-\beta)) = \tau(uv) - \beta \tau(u) - \alpha \tau(v) + \alpha \beta$$
and hence $\tau(uv) = \tau(u)\tau(v)$ for any  freely independent pair of unitaries $u,v \in {\rm U}(M)$. We aim to compute the trace of $uvu^*v^*$. Again, by freeness, we obtain:
\begin{eqnarray*}
0 &=& \tau((u-\alpha)(v-\beta)(u^* - \bar \alpha)(v^* - \bar \beta)) \\
&=& \tau(uvu^*v^*)  - \alpha \tau(vu^*v^*) - \beta \tau(uu^*v^*) - \bar \alpha \tau(uvv^*) - \bar \beta \tau(uvu^*) \\
&& + \alpha \beta \tau(u^*v^*) + |\alpha|^2 \tau(vv^*) + \alpha \bar \beta \tau(vu^*)  + \bar \alpha \beta \tau(uv^*) + |\beta|^2 \tau(uu^*) + \bar \alpha \bar \beta \tau(uv) \\
&& - |\alpha|^2 \beta \tau(v^*) - \alpha |\beta|^2 \tau(u^*) - |\alpha|^2 \bar \beta \tau(v) - \bar \alpha |\beta|^2 \tau(u)  + |\alpha|^2|\beta|^2 \\
&=&  \tau(uvu^*v^*) - 2|\alpha|^2 - 2 |\beta|^2  + 4 |\alpha|^2  |\beta|^2 + |\alpha|^2 + |\beta|^2  - 4 |\alpha|^2  |\beta|^2 + |\alpha|^2  |\beta|^2 \\
&=& \tau(uvu^*v^*) - |\alpha|^2  - |\beta|^2 + |\alpha|^2  |\beta|^2.
\end{eqnarray*}
Hence, we conclude that $$\tau(uvu^*v^*) = |\alpha|^2  + |\beta|^2 - |\alpha|^2  |\beta|^2 = 1 - (1-|\tau(u)|^2)(1-|\tau(v)|^2)$$ for any  freely independent pair of unitaries $u,v \in {\rm U}(M)$. 
On ${\rm U}(M)$, we consider the length function $\ell(u):= \|1-u\|_2 = \sqrt{2-2 \Re(\tau(u))}$.
Since
$$(1-|\tau(u)|^2) = (1+ |\tau(u)|)(1-|\tau(u)|) \leq 2 (1-|\tau(u)|) \leq 2 - 2\Re(\tau(u)) = \ell(u)^2,$$we obtain
\begin{eqnarray*}
\ell([u,v]) &=& \sqrt{2-2 \Re(\tau(uvu^*v^*))} \\
&=& \sqrt{2(1-|\tau(u)|^2)(1- |\tau(v)|^2)} \\
&\leq&\sqrt{2} \cdot \ell(u) \ell(v)
\end{eqnarray*} for any  freely independent pair of unitaries $u,v \in {\rm U}(M)$. With only little more work, we also obtain a lower bound for $\ell([u,v])$.
Indeed, define the projective length function $$\bar \ell(u) = \inf_{|\lambda|=1 } \|\lambda - u\|_2 = \sqrt{2(1-|\tau(u)|)},$$ which metrizes ${\rm PU}(M)$ in a natural way. Note that
$\mu - \mu^2/4 \geq 1/2 \cdot \mu$ for $\mu \in [0,2]$. If $u,v$ are freely independent unitaries, we get:
\begin{eqnarray*}
\ell([u,v])^2 &=& 2(1 - \Re \tau([u,v])) \\
&=& 2 (1-|\tau(u)|^2)(1- |\tau(v)|^2) \\
&=& 2 (1-(1-\bar\ell(u)^2/2)^2) (1-(1-\bar\ell(v)^2/2)^2 \\
&=& 2 (  \bar \ell(u)^2 - \bar \ell(u)^4/4) ( \bar \ell(v)^2 - \bar \ell(v)^4/4) \\
&\geq& 2 (1/2)^2 \cdot \ell(u)^2 \ell(v)^2
\end{eqnarray*}
Hence, we conclude that
$$\ell([u,v]) \geq 1/\sqrt{2} \cdot \bar \ell(u) \bar \ell(v)$$
and we may summarize the computations as follows:
\begin{theorem}
Let $u,v \in {\rm U}(M)$ be freely independent unitaries. Then, we have
$$1/\sqrt{2} \cdot \bar \ell(u) \bar \ell(v) \leq \bar\ell([u,v])= \ell([u,v]) \leq \sqrt{2} \cdot \bar \ell(u) \bar \ell(v) \leq \sqrt{2} \cdot \ell(u) \ell(v).$$
\end{theorem}

For the rest of the paper we fix a sequence of words $(w_n)_n$ in $F_2= \langle x,y \rangle$ such that $w_1=x$ and inductively $w_{n+1} = [w_n,y^{n}xy^{-n}],$ i.e.\ $w_2=[x,yxy^{-1}], w_3=[[x,yxy^{-1}],y^2xy^{-2}],$ and so on. If $u,v$ are freely independent and $v$ is a Haar unitary (i.e. a unitary with $\tau(v^n)=0$ for all $n \in \mathbb Z \setminus \{0\}$), it is easy to see that $ u,vuv^*,v^2u(v^*)^2,\dots,$ are freely independent too. We consider the sequence $(w_n(u,v))_n$, for $i \in \mathbb N$, i.e.
$$u_1=u, u_2=[u,vuv^*], u_3=\left[[u,vuv^*],v^2u(v^*)^2 \right], \dots.$$
Thus, we obtain from the previous theorem by induction
$$(1/\sqrt{2})^{n-1} \cdot \bar\ell(u)^n \leq \ell(w_n(u,v)) \leq (\sqrt{2})^{n-1} \cdot \ell(u)^n.$$

The upper bound was rather surprising for us at first, since $\ell(u)< 1/\sqrt{2}$ does not seem to be a severe assumption and we may assume in addition that the subgroup generated by $u$ in ${\rm PU}(M)$ is discrete, we may even assume that $u'$ is a Haar unitary in a corner $pMp$ and $u = u' + p^{\perp}$ for a projection $p$ with $\tau(p)<1/4$.

\begin{corollary} Let $M$ be a tracial von Neumann algebra. Let $u,v \in {\rm U}(M)$ be freely independent unitaries with $\ell(u)<1/\sqrt{2}$ and $v$ a Haar unitary. For every $\varepsilon>0$, there exists $w \in \mathbb F_2$, such that $\ell(w(u,v))< \varepsilon$.
\end{corollary}

\begin{corollary}
Let $M$ be a separable ${\rm II}_1$-factor. There exists a hyperfinite ${\rm II}_1$-factor $R \subset M$ and Haar unitaries $v_1,\dots,v_n,\dots$ in ${\rm U}(R)$ such that the following holds. For every $u \in {\rm U}(M)$ satisfying $\ell(u)< 1/\sqrt{2}$ and for every $\varepsilon>0$, there exists $w \in F_2$ and $n \in \mathbb N$ such that $\ell(w(u,v_n))< \varepsilon$.
\end{corollary}
\begin{proof}
By work of Popa \cite{MR1372533}, we may assume that $R' \cap M= \mathbb C$ and that $v = [(v_n)_n] \in {\rm U}(M^{\omega})$ is free from $M$ inside $M^{\omega}$. Now apply the previous theorem and pass to some $v_n$ along the ultralimit. \end{proof}

\begin{remark}
Sources of increasing sequences of discrete groups in ${\rm PU}(M)$ that become less and less discrete come from many sources. For example, it is easy to see that $\mathbb Z \subset {\rm PU}(L\mathbb Z)$ is not maximal discrete, in fact ${\rm U}(L\mathbb Z)$ does not admit a maximal discrete subgroup at all. Another example is $\cup_n {\rm Sym}(2^n) \subset {\rm U}(\otimes_n M_2(\mathbb C))$. More artificially, we may consider ${\rm U}(LF_2)$ and pick a partition of $1$ into orthogonal projections $(p_n)_n$ of trace $2^{-n}$. Then
$$\prod_{n\geq 1} {\rm U}((LF_2)_{2^{-n}}) \subset {\rm U}(LF_2)$$
and $(LF_2)_{2^{-n}} = LF_{1 + 2^{2n}}$ by results of Voiculescu \cite{MR1217253}, so that $\bigoplus_{n \geq 1} F_{1+2^{2n}}$ is a subgroup of ${\rm U}(LF_2)$. This last example also shows clearly, that there is no reason to expect any form of commutator contractivity near the identity in general.

\end{remark}

\section{Applications to group von Neumann algebras}

We will now turn our attention to group von Neumann algebras. Let $G$ be a group.  A sequence $(g_n)_n$ of elements in $G$ is called asymptotically free if $(g_n)_n$ is free from the diagonal copy of $G$ in $\prod_n G$. Equivalently, for every $k \in \mathbb N$, $s_1,\dots,s_k \in G \setminus \{1\}$ and every $\varepsilon_1,\dots,\varepsilon_k \in \mathbb Z \setminus \{0\}$, there exists $n \in \mathbb N$, such that
\[
s_1 g_n^{\varepsilon_1} s_2 g_n^{\varepsilon_2} s_3 \cdots g_n^{\varepsilon_k} \neq 1
\]
Groups admitting an asymptotically free sequence have been studied for a long time.

Recall that a mixed identity for $G$ is a word $w \in \mathbb Z \ast G = \langle G, t\rangle$ in one variable $t$ with coefficients in $G$ such that $w$ evaluates to the identity if any element of $G$ is substituted for $t$. A group is called mixed identity free (MIF) if it does not satisfy any non-trivial mixed identity. We refer the reader to \cite{MR3556970} for further details around this notion, emphasizing the following statement.

\begin{proposition}[{\cite[proof of Proposition 5.3 and Remark 5.1]{MR3556970}}]
Let $G$ be a group. The following are equivalent:
\begin{enumerate}
\item The group $G$ contains an asymptotically free sequence.
\item There is no mixed identity $w \in \mathbb Z \ast G$ for $G$.
\item $G$ is mixed identity free, i.e.\ there exists no mixed identity in $\mathbb F_\infty \ast G$.
\end{enumerate}
\end{proposition}

We denote the group von Neumann algebra of $G$ with its natural trace by $(LG,\tau)$. It is well-known that any non-trivial conjugacy class in a ${\rm MIF}$~group is infinite. In particular, the group von Neumann algebra $LG$ of a ${\rm MIF}$~group $G$ is a ${\rm II}_1$-factor. Note that the left-regular representation provides a natural inclusion $G \subset {\rm PU}(LG)$.
 We will be interested in the existence if maximal discrete subgroups of ${\rm PU}(LG)$ containing $G$.

\begin{lemma}
Let $(g_n)_n$ be an asymptotically free sequence. There exists an ultrafilter $\omega$, such that $[(g_n)_n] \in U((LG)^{\omega})$ is freely independent from the diagonal copy of $LG$ in $(LG)^\omega$. 
\end{lemma}
This lemma is standard but not completely trivial, and we include the argument for the sake of completeness.
\begin{proof}
Let $(w_m)$ be an enumeration of all mixed identities, that is, elements in $\mathbb Z\ast G$. We denote by $w(m_1,\dots,m_\ell)$ the iterated commutator
\[
w(m_1,\dots,m_\ell) = [w_{m_1},[w_{m_2},\dots [w_{m_{\ell-1}},w_{m_\ell}]]].
\]
Observe that if $w(m_1,\dots,m_\ell)(g) \neq 1$, it guarantees that $w_{m_i}(g)\neq 1$ for all $i=1,\dots,\ell$. Let $v_\ell\coloneqq w(1,2,\dots,\ell)$. Asymptotic freeness of $(g_n)$ guarantees the existence of natural numbers $n_\ell$ such that $v_\ell(g_{n_\ell})\neq 1$. Taking an ultrafilter containing this sequence finishes the proof.
\end{proof}

\begin{lemma}
Let $G$ be a ${\rm MIF}$~group and $u \in {\rm U}(LG)$ with $u \not \in S^1 \cdot 1$ and $\Re(\tau(u))>3/4$. For any $\varepsilon>0$ there exists some $w \in \mathbb F_2$ and $g \in G$ such that $\ell(w(u,g))< \varepsilon$. Moreover, the image of the map from $\mathbb Z \ast G$ to ${\rm PU}(LG)$ that maps the generator $1 \in \mathbb Z$ to $u$ is not discrete.
\end{lemma}
\begin{proof}
Note that from the assumptions $\bar \ell(u) \neq 0$ and $\ell(u)< \sqrt{2}$.
Let $(g_k)_k$ be an asymptotically free sequence and set $v= [(g_k)_k] \in (LG)^{\omega}$. By assumption $\ell(u) \leq 1/\sqrt{2} - \delta$ for some $\delta>0$. By the results from above, we get
$$(1/\sqrt{2})^{n-1} \cdot \bar\ell(u)^n \leq \ell(w_n(u,v)) \leq(\sqrt{2})^{n-1} \cdot \ell(u)^n.$$
Thus, for some $k \in \mathbb N$
$$(1/\sqrt{2})^{n} \cdot \bar\ell(u)^n \leq \ell(w_n(u,g_k)) \leq (\sqrt{2})^{n} \cdot \ell(u)^n$$
and this finishes the proof.
\end{proof}

The following consequence is immediate.
\begin{corollary}
Let $G$ be a ${\rm MIF}$~group and let $u \in {\rm PU}(LG)$ with $0 < \bar \ell(u) < 1/\sqrt{2}$. Then, the subgroup $\langle u, G \rangle \subset {\rm PU}(LG)$ is not discrete.
In particular, the set of discrete subgroups of ${\rm PU}(LG)$ containing $G$ is uniformly discrete.
\end{corollary}
\begin{proof} By the previous lemma, no such subgroup contains a non-trivial element with $\bar \ell(u)<1/\sqrt{2}$. Indeed, we would apply the previous lemma to a suitable lift $u'$ with $\ell(u')< 1/\sqrt{2}$.\end{proof}

We are now able to state and prove our main result.
\begin{theorem} \label{main}
Let $G$ be a ${\rm MIF}$~group. Then, $G$ is contained in a maximal discrete subgroup in ${\rm PU}(LG)$. 
\end{theorem}
\begin{proof}
The discrete subgroups of ${\rm PU}(LG)$ containing $G$ are ordered by inclusion.   Now, the result follows by Zorn's lemma since unions over chains remain discrete.
\end{proof}

\begin{corollary}
Let $R$ be the hyperfinite ${\rm II}_1$-factor. The group ${\rm PU}(R)$ contains a maximal discrete subgroup.
\end{corollary}
\begin{proof}
It is known by results of Hull and Osin \cite{MR3556970} that there are amenable ${\rm MIF}$~groups. Any ${\rm MIF}$~group is automatically an i.c.c. group (that is, it has only infinite non-trivial conjugacy classes). See also \cite{MR4193626} for an example of an elementary amenable group that is ${\rm MIF}$. Since $LG=R$ in the i.c.c.\ amenable case by Connes \cite{MR454659}, this finishes the proof.
\end{proof}

\begin{question}
When is $G \subset {\rm PU}(LG)$ a maximal discrete subgroup?
\end{question}

Due to existence of central subgroups, the preceding question is only interesting for i.c.c.\ groups. However, we do not have a single i.c.c.\ group, where we could decide this question.

\section{Miscellaneous results}

In this section, we start out by briefly discussing maximal discrete subgroups in unitary groups of $C^*$-algebras. The following result is in analogy to the Jordan-Schur theorem mentioned in the introduction.

\begin{theorem}\label{cstar}
Let $A$ be a unital $C^*$-algebra and $\Gamma$ be a discrete subgroup of its unitary group ${\rm U}(A)$. Then, the subgroup
$$\Gamma_{1/2} = \langle g \in \Gamma \mid \|1-g\|<1/2 \rangle$$
is abelian and normal in $\Gamma$.
\end{theorem}
\begin{proof}
We set $\ell(g)=\|1-g\|$ and note that
\begin{eqnarray*} \ell([g,h]) &=& \|1-ghg^*h^*\| = \|hg-gh\| \\
&=& \|(1-h)(1-g)-(1-g)(1-h)\| \leq 2 \|1-g\|\|1-h\| = 2 \ell(g) \ell(h).
\end{eqnarray*}
We define $\Gamma_{t} := \langle g \in \Gamma \mid \ell(g)<t \rangle$. 
Fix $\varepsilon>0$ and set
$$\delta = \inf\{\ell(g) \mid g \in \Gamma_{1/2-\varepsilon} \mbox{ is not central in } \Gamma_{1/2-\varepsilon}\}.$$
Suppose that this infimum is positive and finite, i.e.\ that the set of non-central elements is bounded away from $1$ and non-empty. If $g$ is non-central in $\Gamma_{1/2-\varepsilon}$, then it cannot commute with the entire generating set of $\Gamma_{1/2-\varepsilon}$. Thus, we can consider $g,h \in \Gamma_{1/2-\varepsilon}$ with $\ell(g) < (1+2\varepsilon)\delta$, $\ell(h)<1/2 - \varepsilon$ and $[g,h] \neq 1$. We compute $$\ell([g,h]) \leq 2 (1+2\varepsilon)\delta (1/2-\varepsilon)<\delta$$ and hence $[g,h]$ is central in $\Gamma_{1/2-\varepsilon}$, in particular, it commutes with $g$ and $h$. Thus, we obtain a unitary representation of the Heisenberg group
$$H(\mathbb Z) = \langle x,y,z \mid [x,y]=z, [x,z]=[y,z]=1 \rangle.$$ But for any unitary representation of the Heisenberg group for which $z$ acts non-trivially, both generators are far away from the identity. Indeed, irreducible unitary representations are parametrised by $\theta \in S^1$ and either lead to the unique irreducible representation of the non-commutative torus $A_{\theta}$ or they are finite-dimensional and completely understood. In any case, the generators satisfy $\|1-g\| \geq \sqrt{3}$ and $\|1-h\| \geq \sqrt{3}$ unless they commute. This implies that $\Gamma_{1/2-\varepsilon}$ is abelian. Since $\varepsilon>0$ was arbitrary and $$\Gamma_{1/2} = \bigcup_{\varepsilon>0} \Gamma_{1/2-\varepsilon},$$
this finishes the proof.
\end{proof}

The following corollary is immediate.

\begin{corollary}
Let $A$ be a unital $C^*$-algebra and $\Gamma \subset {\rm U}(A)$ be a discrete subgroup. Then, at least one of two conditions are satisfied:
\begin{enumerate}
    \item The set of discrete subgroups containing $\Gamma$ is uniformly discrete. In particular, the group $\Gamma$ is contained in a maximal discrete subgroup.
    \item The group $\Gamma$ normalizes a non-trivial abelian subalgebra of $A$, and there exists an ascending chain of subgroups of $U(A)$ containing $\Gamma$, which is not uniformly discrete.
\end{enumerate}
\end{corollary}

It is natural to wonder about maximal finite subgroups of compact Lie groups, most notably of ${\rm PU}(n)$. The following result is elementary and surely known to experts.
\begin{proposition} \label{compact}
Let $G$ be a finite group. Let $\pi \colon G \to {\rm SU}(n)$ be a non-trivial irreducible representation and let $\bar \pi \colon G \to {\rm PU}(n)$ be the corresponding projective representation. Assume that $\pi(G)$ does not normalize a non-trivial torus in ${\rm SU}(n)$, equivalently, every $\pi(G)$-invariant abelian Lie subalgebra of $\mathfrak{su}(n)$ is trivial.
In particular, this holds when $G$ does not admit a non-trivial homomorphism to ${\rm Sym}(n)$, or, more generally,
when $\pi$ is non-trivial and irreducible of least dimension.
    
Then, the set of discrete subgroups of ${\rm PU}(n)$ containing $\bar\pi(G)$ is uniformly discrete. In particular, $\bar\pi(G)$ is contained in a maximal discrete subgroup.
\end{proposition}
\begin{proof} Assume that there exists a sequence of finite subgroups $H_0, H_1, H_2, \cdots$ containing $\pi(G)$ which is getting less and less discrete in ${\rm SU}(n)$. By the Jordan--Schur theorem we get a sequence of abelian normal subgroups $A_i \subset H_i$ of uniformly bounded index. Consider an ultralimit $A$ of the sequence $(A_i)_i$ in the space of closed subgroups of ${\rm SU}(n)$ and let $A_0$ be the connected component of the identity of $A$. Note that $A_0$ is normalized by $\pi(G)$. In order to obtain a contradiction, we will show that $A_0$ is trivial.
We denote the Lie algebra of $A_0$ by $\mathfrak{a}$, which is naturally acted upon by  the group $G$. We conclude that $\mathfrak{a}$ is trivial and hence $A_0$ is trivial. This finishes the proof.

For the proof of the addendum, note that ${\rm dim}_{\mathbb R} (\mathfrak{a}) \leq n-1$, so that if $\pi$ was irreducible of least dimension, the action of $\pi(G)$ on $\mathfrak{a}$ must be trivial. Thus, if $\mathfrak{a}$ is non-trivial, we obtain a non-scalar matrix commuting with $\pi(G)$, contrary to our assumption that $\pi$ was irreducible.
\end{proof}

Let us remark that by classification of finite subgroups of $\rm{SO}(3) \cong \rm{PU}(2)$ the abelian rotation subgroup $\mathbb Z/n$ is not contained in any maximal finite subgroup for $n > 5$: the only finite subgroups $\rm{SO}(3)$ which come in question are dihedral subgroups, and one can always embed a dihedral subgroup into a bigger dihedral subgroup.

\section*{Acknowledgments}

This research was supported by the ERC Consolidator Grant No.\ 681207.

\end{document}